\documentclass[11pt]{amsart}
\usepackage{amsfonts,amssymb,amsmath,amsthm}
\usepackage{url}
\usepackage{enumerate}
\usepackage[dvipdfmx]{graphicx}

\newtheorem{theorem}{Theorem}[section]

\newtheorem{lemma}[theorem]{Lemma}

\newtheorem{proposition}[theorem]{Proposition}
\newtheorem{corollary}[theorem]{Corollary}
\theoremstyle{definition}
\newtheorem{definition}[theorem]{Definition}

\numberwithin{equation}{section}

\author{Shohei Satake}
\address{
Graduate School of System Informatics, Kobe University \\
Rokkodai 1-1, Nada, Kobe, 657-8501, JAPAN}
\email{155x601x@stu.kobe-u.ac.jp}
\thanks{The author is supported by Grant-in-Aid for JSPS Fellows 18J11282 of the Japan Society for the Promotion of Science.}



\begin{document}

\title[A constructive solution to a problem of ranking tournaments]
{A constructive solution to a problem of ranking tournaments}

\maketitle

\begin{abstract}
A tournament is an oriented complete graph. 
The problem of ranking tournaments was firstly investigated by P. Erd\H{o}s and J. W. Moon.
By probabilistic methods, the existence of ^^ ^^ unrankable'' tournaments was proved.
On the other hand, they also mentioned the problem of explicit constructions. 
However, there seems to be only a few of explicit constructions of such tournaments.
In this note, we give a construction of many such tournaments by using skew Hadamard difference sets which have been investigated in combinatorial design theory.
\end{abstract}

\section{Introduction} 
For a digraph $D$, let $V(D)$ and $E(D)$ be the vertex and the edge set of $D$, respectively. 
And for two distinct vertices $x$ and $y$, let the ordered pair $(x, y)$ denote the edge directed from $x$ to $y$. 
Let $T$ be a tournament with $n$ vertices and let $\sigma$ be a bijection from $V(T)$ to $\{1, 2, \ldots, n\}$.
An edge $(x, y)$ of $T$ is called {\it consistent} with $\sigma$ if $\sigma(x)<\sigma(y)$. 
$C(T, \sigma)$ is defined as the number of consistent edges with $\sigma$ and $C(T):=\max_{\sigma}C(T, \sigma)$.
These concepts was from paired comparisons in statistics (see e.g. \cite{KS40}). 
It is reasonable to find the most suitable rankings, that is, the bijection with the maximum number of consistent edges.
First we see that for every tournament $T$ with $n$ vertices,
\[
\frac{1}{2}\binom{n}{2} \leq C(T)\leq \binom{n}{2}.
\] 
The first inequality follows by
\begin{equation}
\label{eq:plus}
C(T, \sigma)+C(T, \sigma')=\binom{n}{2},
\end{equation}
where $\sigma'$ is the reversed ranking of $\sigma$ which is defined as $\sigma'(v):=n-\sigma(v)+1$ for each $v \in V(T)$.
And in the second inequality, the equality holds if and only if $T$ is a transitive tournament.
Thus it seems to be natural to consider the worst case. 
In \cite{EM65}, it was proved that for any $\varepsilon>0$, random tournaments $\mathcal{T}_n$ with $n$ vertices satisfy the following property with high probability:
\[
C(\mathcal{T}_n)\leq \Bigl(\frac{1}{2}+\varepsilon \Bigr)\binom{n}{2}.
\]
Moreover, Spencer~\cite{S71}, \cite{S80} and de la Vega~\cite{d83} proved that for sufficiently large $n$, there exist some constant numbers $c_1$ and $c_2$ such that
\begin{equation}
\label{eq:prob}
\frac{1}{2}\binom{n}{2}+c_1n^{\frac{3}{2}} \leq \min_T C(T) \leq \frac{1}{2}\binom{n}{2}+c_2n^{\frac{3}{2}}
\end{equation}
where the minimum is taken over all tournaments with $n$ vertices.
\par The problem constructing explicit such tournaments was mentioned in Erd\H{o}s-Moon~\cite{EM65} and Spencer~\cite{S85}.
However, at this point, there seems to be almost no explicit construction of such tournaments except for Paley tournaments.
Here, for a prime $p \equiv 3 \pmod{4}$, the {\it Paley tournament} $T_p$ is the tournament with vertex set $\mathbb{F}_p$, the finite field of $p$ elements, 
and edge set formed by all edges $(x, y)$ such that $x-y$ is a non-zero square of $\mathbb{F}_p$. 
In \cite[Theorem 9.1.1]{AS16}, it was proved that 
\begin{equation}
\label{eq:paley}
C(T_p) \leq \frac{1}{2}\binom{p}{2}+p^{\frac{3}{2}}\log_2 (2p).
\end{equation}
As shown in \cite{A06}, such explicit examples can be applied, for example, for a derandomized proof of the NP-hardness of the feedback arc set problem for tournaments 
which was firstly proved under randomized reduction in \cite{ACN08}.      
\par In this note, we give a generalized construction of such tournaments.  
We note that the proof of (\ref{eq:paley}) in \cite{A06} and \cite{AS16} contains a discussion which can be applied only for Paley tournaments.
In the present author's paper~\cite{S18}, we generalize that discussion to more general cases by focusing on digraph spectra and a digraph-version of the expander-mixing lemma. 
Here we focus on doubly regular tournaments and slightly improve the result for doubly regular tournaments in \cite{S18} by an alternative proof. 
Moreover, we give exponentially many non-isomorphic doubly regular tournaments obtained by a known construction of skew Hadamard difference sets.  

\section{Ranking doubly regular tournaments}
In this section, we show that doubly regular tournaments are desired tournaments.
This was also mentioned in the present author's paper~\cite{S18}.
Here we give a different proof.  
\par First we give the definition of doubly regular tournaments.
A digraph $D$ is called $d$-regular if in-degree and out-degree of each vertex is $d$.
And for two distinct vertices $x$ and $y$, let $N^+(x, y)$ (resp. $N^-(x, y)$) be the set of vertices $z$ such that $(x, z) , (y, z) \in E(D)$ (resp. $(z, x) , (z, y) \in E(D)$).
\begin{definition}
A tournament $T$ with $n$ vertices is called a {\it doubly regular tournament} if $T$ is $(n-1)/2$-regular and 
for any distinct two vertices $x$ and $y$, $|N^+(x,y)|=|N^-(x, y)|=(n-3)/4$.
\end{definition}
We basically use the discussion in \cite[Section 9.1]{AS16} but we need to show a new upper bound of $e(A, B)-e(B, A)$ 
where for a digraph $D$ and disjoint subsets $A, B \subset V(D)$, $e(A, B)$ is defined as
\begin{equation}
e(A, B):=\bigl|\{(a, b)\in E(D) \mid a \in A, \; b \in B\} \bigr|.
\end{equation}
In fact, the upper bound in \cite{AS16} holds only for Paley tournaments 
(the proof uses the properties of the quadratic residue character). 
To obtain such bound, we need to consider the adjacency matrix of a digraph.
The {\it adjacency matrix} $M_D$ of a digraph $D$ with vertices is the $\{0, 1\}$-square matrix of size $n$ 
whose rows and columns are indexed by the vertices of $D$ and the $(x, y)$-entry is equal to $1$ if and only if $(x, y) \in E(D)$.
Now we are ready to show our desired bound for doubly regular tournaments.

\begin{lemma}
Let $T$ be a doubly regular tournament with $n$ vertices.
Then for any disjoint two subsets $A, B \subset V(T)$, 
\begin{equation}
\label{eq:e}
e(A, B)-e(B, A) \leq \sqrt{n \cdot |A| \cdot |B|}.
\end{equation}
\end{lemma}
We remark that when $T$ is $T_p$, (\ref{eq:e}) coincides to the bound of Lemma 9.1.2 in \cite{AS16}.
Since Paley tournaments are doubly regular tournaments, this lemma gives an generalization of the bound in \cite{AS16}.
Moreover, this lemma improves the bound obtained by Corollary 2.3 and Lemma 3.6 in the present author's paper \cite{S18}.

\begin{proof}
Let $M:=M_T$ and $\tilde{M}:=2M-(J_n-I_n)$ where $I_n$ and $J_n$ are the identity matrix and the all-one matrix of order $n$, respectively.
That is, $\tilde{M}$ is the $\{0, \pm 1\}$-matrix obtained from $M$ by replacing all non-diagonal $0$-entries by $-1$.
Let $m_{i j}$ be the $(i, j)$-entry in $\tilde{M}$. 
By the definition of $\tilde{M}$, $m_{i j}=1$ if $(i, j) \in E(T)$ and $m_{i j}=-1$ if $(j, i) \in E(T)$.
Thus for any disjoint two subsets $A, B \subset V(T)$, we see that 
\begin{equation}
\label{eq:e1}
e(A, B)-e(B, A)=\sum_{i \in A}\sum_{j \in B}m_{ij}.
\end{equation}
By the Cauchy-Schwarz inequality, 
\begin{align*}
\begin{split}
\Bigl(\sum_{i \in A}\sum_{j \in B}m_{ij} \Bigr)^2 &\leq |A| \sum_{i \in A} \Bigl(\sum_{j \in B}m_{ij} \Bigr)^2\\
&\leq |A| \sum_{i \in V(T)} \Bigl(\sum_{j \in B}m_{ij} \Bigr)^2\\
&=|A| \sum_{i \in V(T)} \Bigl(|B|+2\sum_{j<l \in B}m_{ij}m_{il} \Bigr)\\
&=|A||B|n+2|A|\sum_{j<l \in B}\sum_{i \in V(T)} m_{ij}m_{il}. 
\end{split}
\end{align*}
The idea of this inequality can be found in the proof of Lemma 9.1.2 in \cite{AS16}.  
Now we can show that for every $1 \leq j\neq l \leq n$,
\begin{equation}
\label{eq:e2}
\sum_{i \in V(T)} m_{ij}m_{il}=-1.
\end{equation}
In fact, from the definition of doubly regular tournaments, we see that
\begin{equation*}
MM^T=\frac{n+1}{4}I_n+\frac{n-3}{4}J_n.
\end{equation*} 
And since $T$ is a tournament and so $M+M^T=J_n-I_n$, a simple calculation shows that 
\begin{equation*}
\tilde{M} \tilde{M}^T=nI_n-J_n.
\end{equation*}
So, by (\ref{eq:e2}), we see that
\begin{equation}
\label{eq:e3}
\Bigl(\sum_{i \in A}\sum_{j \in B}m_{ij} \Bigr)^2 \leq |A||B|n-|A||B|(|B|-1) \leq |A||B|n.
\end{equation}
Thus by (\ref{eq:e1}) and (\ref{eq:e3}), we obtain the lemma.
\end{proof}
\par From this lemma and the argument in \cite[pp.150-151]{AS16}, we get the following lemma. 
\begin{lemma}
\label{cor:exp2}
Let $T$ be a doubly regular tournament with $n$ vertices. Then
\begin{equation}
C(T, \sigma)-C(T, \sigma') \leq  n^{\frac{3}{2}} \log_2 (2n).
\end{equation}
\end{lemma}
By (\ref{eq:plus}) and Lemma~\ref{cor:exp2}, we immediately obtain the following theorem which gives a generalization of (\ref{eq:paley}).
\begin{theorem}
\label{thm:DRT}
Let $T$ be a doubly regular tournament with $n$ vertices. Then,
\begin{equation}
\label{eq:DRT}
C(T)\leq \frac{1}{2}\binom{n}{2}+n^{\frac{3}{2}} \log_2 (2n).
\end{equation}
\end{theorem}
We remark that this theorem can not give the asymptotically best possible upper bound of $|\min_{T}C(T)-\binom{n}{2}/2|$ obtained by (\ref{eq:prob}).
However, by this theorem, we see that all doubly regular tournaments give the best known constructive upper bound obtained by (\ref{eq:paley}).

\section{Doubly regular tournaments from skew Hadamard difference sets}
In this section, we explain that doubly regular tournaments can be obtained from skew Hadamard difference sets.
We also give exponentially many non-isomorphic examples by using a known construction of skew Hadamard difference sets.
At first, we give the definition of skew Hadamard difference sets.
\begin{definition}
Let $\Gamma$ be an abelian group of order $n$. We denote the operation additively and let $0$ be the identity. 
Then, a subset $D \subset \Gamma \setminus \{0\}$ is called an {\it Hadamard difference set} in $\Gamma$ if 
$|D|=(n-1)/2$ and for each $g \in \Gamma \setminus \{0\}$, $g$ appears exactly $(n-3)/4$ times 
in the sequence $(d_1-d_2)_{d_1, d_2 \in D,\; d_1\neq d_2}$.
An Hadamard difference set $D$ in $\Gamma$ is called {\it skew} if $\Gamma=\{0\} \sqcup D \sqcup -D$
where $-D=\{-d \mid d \in D\}$ and $A \sqcup B$ denotes the disjoint union of $A$ and $B$.
\end{definition}
\par Next we explain Cayley digraphs and show that Cayley digraphs defined by skew Hadamard difference sets are doubly regular tournaments.
\begin{definition}
Let $D$ be a subset of $\Gamma \setminus \{0\}$. The {\it Cayley digraph} $Cay(\Gamma, D)$ over $\Gamma$ defined by $D$ 
is the digraph with vertex set $\Gamma$ such that for two vertices $x$ and $y$, $x \rightarrow y$ if $x-y \in D$.
\end{definition}

\begin{proposition}
\label{prop:cayley}
$Cay(\Gamma, D)$ is a doubly regular tournaments if and only if $D$ is a skew Hadamard difference set in $\Gamma$.
\end{proposition}
\begin{proof}
It is not so difficult to check that $Cay(\Gamma, D)$ is a tournament if and only if $\Gamma=\{0\} \sqcup D \sqcup -D$.
And from the definition of Cayley digraphs, $Cay(\Gamma, D)$ is $|D|$-regular.
Moreover, we see that for each two distinct vertices $x$ and $y$,
\[
|N^+(x, y)|=|(x-D) \cap (y-D)|=|(x+D) \cap (y+D)|=|N^-(x, y)|,
\]
where $x+D=\{x+d \mid d\in D\}$ and $x-D=\{x-d \mid d\in D\}$. 
It is not so hard to check that $|(x+D) \cap (y+D)|$ is equal to the frequency of the difference $x-y$ in $(d_1-d_2)_{d_1, d_2 \in D,\; d_1\neq d_2}$, proving the proposition.
\end{proof}
\par The constructing problem of skew Hadamard difference sets has been investigated in combinatorial design theory.
Until 2006, there had been no construction of skew Hadamard difference sets except for Paley difference sets which give Paley tournaments under Proposition~\ref{prop:cayley}.
In 2006, Ding-Yuan~\cite{DY06} constructed new infinite families of skew Hadamard difference sets in the additive group of the finite field $\mathbb{F}_{3^m}$ such that $m\geq 3$ and $m$ is odd. 
After their work, other constructions were obtained, for example, in \cite{DPW15}, \cite{FX12}, \cite{FMX15}, \cite{M13} and \cite{M10}.
Especially, in \cite{M10}, the author constructed exponentially many inequivalent skew Hadamard difference sets in the additive group $(\mathbb{Z}/p\mathbb{Z})^3$ for each prime $p \equiv 3 \pmod{4}$. 
\par Below, we prove that from the construction in \cite{M10}, we can obtain exponentially many non-isomorphic doubly regular tournaments.
Let $D_1$ and $D_2$ be skew Hadamard difference sets in $\Gamma$. Then, $D_1$ and $D_2$ are called {\it equivalent} if
there is a group automorphism $\tau$ of $\Gamma$ and an element $g \in \Gamma$ such that $D_1=\tau(D_2)+g$.
In general, $Cay(\Gamma, D_1)$ may be isomorphic to $Cay(\Gamma, D_2)$ even if $D_1$ and $D_2$ are inequivalent.
However, if $\Gamma$ is a CI-group for Cayley digraphs, the inequivalence of $D_1$ and $D_2$ implies that $Cay(\Gamma, D_1)$ and $Cay(\Gamma, D_2)$ are non-isomorphic.
Here $\Gamma$ is called {\it CI-group for Cayley digraphs} if for all two subsets $D_1$ and $D_2$, when $Cay(\Gamma, D_1)$ and $Cay(\Gamma, D_2)$ are isomorphic, 
there exists a group automorphism of $\Gamma$ which is an isomorphism between $Cay(\Gamma, D_1)$ and $Cay(\Gamma, D_2)$.

\begin{proposition}
\label{prop:ineq}
Let $\Gamma$ be a CI-group for Cayley digraphs and $D_1$ and $D_2$ inequivalent skew Hadamard difference sets in $\Gamma$.
Then $Cay(\Gamma, D_1)$ is not isomorphic to $Cay(\Gamma, D_2)$.
\end{proposition}
\begin{proof}
Assume that $Cay(\Gamma, D_1)$ is isomorphic to $Cay(\Gamma, D_2)$.
Then, there must exist a group automorphism $\tau$ of $\Gamma$ which is an isomorphism from $Cay(\Gamma, D_1)$ to $Cay(\Gamma, D_2)$.
Since $\tau$ is a group automorphism, for any $x, y \in \Gamma$, $\tau(x-y)=\tau(x)-\tau(y)$.
And since $\tau$ is also an isomorphism, $x-y \in D_1$ implies that $\tau(x)-\tau(y) \in D_2$.
So $\tau(D_1) \subset D_2$.
Now for $d \in D_2$, let $z, w$ be vertices such that $d=z-w$ and so $(z, w)$ is an edge of $Cay(\Gamma, D_2)$.
Since $\tau^{-1}$ is an isomorphism from $Cay(\Gamma, D_2)$ to $Cay(\Gamma, D_1)$, $\tau^{-1}(z)-\tau^{-1}(w)$ should be in $D_1$.
And since $\tau^{-1}$ is also a group automorphism, $\tau^{-1}(z)-\tau^{-1}(w)=\tau^{-1}(z-w)$ and so $d=z-w \in \tau(D_1)$.
Thus $\tau(D_1)=D_2$, which contradicts the assumption of $D_1$ and $D_2$.
\end{proof}
\par The following theorem for $(\mathbb{Z}/p\mathbb{Z})^3$ can be found in \cite{AN99} (see also e.g. \cite{L02}). 
\begin{theorem}[Theorem 3.1 in \cite{AN99}]
\label{thm:CI}
Let $p$ be an odd prime. Then the additive group $(\mathbb{Z}/p\mathbb{Z})^3$ is a CI-group for Cayley digraphs.
\end{theorem}

As a direct consequence of Proposition~\ref{prop:ineq}, Theorem~\ref{thm:DRT} and \ref{thm:CI}, we get the following corollary.

\begin{corollary}
\label{cor:exp}
From the construction in \cite{M10}, for each prime $p \equiv 3 \pmod{4}$, we obtain exponentially many non-isomorphic doubly regular tournaments with $p^3$ vertices. 
Moreover, each such doubly regular tournament $T$ satisfies 
\begin{equation*}
C(T)\leq \frac{1}{2}\binom{n}{2}+n^{\frac{3}{2}} \log_2 (2n)
\end{equation*}
where $n=p^3$.
\end{corollary}
We remark that known other constructions are usually in the additive group of a finite field of odd characteristic $p \equiv 3 \pmod{4}$ 
which is the elementary abelian group $(\mathbb{Z}/p\mathbb{Z})^k$ for some $k \geq 1$.
It is known that $(\mathbb{Z}/p\mathbb{Z})^k$ is a CI-group when $k \leq 4$.
However, in general, $(\mathbb{Z}/p\mathbb{Z})^k$ is not a CI-group when $k$ is sufficiently large (see \cite{M03} and \cite{L02}).
Since $p^k$ should be of the form $p^k \equiv 3 \pmod{4}$ and there is no skew Hadamard difference set in a cyclic group except for the Paley's difference set 
(see \cite{J66}), at this point, we can check the isomorphism problem as Proposition~\ref{prop:ineq} only for the case of $k=3$.
    



\end{document}